\documentclass{amsart}
\usepackage[utf8]{inputenc}

\usepackage{amsmath,amsfonts, amssymb}
\usepackage{graphicx}
\usepackage{xcolor}
\usepackage{enumerate}
\usepackage{enumitem}
\usepackage{cite}
\usepackage{tabu}
\usepackage{ulem}

\usepackage{stmaryrd} 

\usepackage{tikz}

\usepackage{fullpage}
\usepackage[active]{srcltx}

\usepackage{changepage}

\usepackage{chngcntr}

\usepackage{multibib}

\usepackage{url}
\usepackage{hyperref}
\hypersetup{
	colorlinks,
	linkcolor={red!50!black},
	citecolor={blue!50!black},
	urlcolor={blue!80!black}
}  

\theoremstyle{plain}
\newtheorem{theorem}{Theorem}[section]
\newtheorem*{theo:sutmanwithnonisolatinggutsishoriprime}{Theorem~\ref{theo:sutmanwithnonisolatinggutsishoriprime}}
\newtheorem*{thm:handlenumberofgutsishandlenumberofmfld}{Theorem~\ref{thm:handlenumberofgutsishandlenumberofmfld}}
\newtheorem*{thm:bothbounds}{Theorem~\ref{thm:bothbounds}}

\newtheorem{cor}[theorem]{Corollary}

\newtheorem{prop}[theorem]{Proposition}
\newtheorem{lemma}[theorem]{Lemma}
\newtheorem{claim}[theorem]{Claim}

\theoremstyle{definition}
\newtheorem{defn}[theorem]{Definition}
\newtheorem{remark}[theorem]{Remark}
\newtheorem{example}[theorem]{Example}

\newcommand{\comment}[1]{}

\newcommand{\bdry}{\ensuremath{\partial}}

\newcommand{\nbhd}{\ensuremath{\mathcal{N}}}

\newcommand{\TN}{{\ensuremath{\rm{TN}}}}

\newcommand{\cut}{\ensuremath{\backslash}}
\newcommand{\ccut}{\ensuremath{\backslash\!\!\backslash}}

\newcommand{\dcsum}{\ensuremath{\mathbin{\rotatebox[origin=c]{-30}{$\asymp$}}}} 
    
\definecolor{amaranth}{rgb}{0.9, 0.17, 0.31} 
\definecolor{carrotorange}{rgb}{0.93, 0.57, 0.13} 
\definecolor{citrine}{rgb}{0.89, 0.82, 0.04} 
\definecolor{dartmouthgreen}{rgb}{0.05, 0.5, 0.06} 
\definecolor{ballblue}{rgb}{0.13, 0.67, 0.8} 
\definecolor{ceruleanblue}{rgb}{0.16, 0.32, 0.75} 
\definecolor{amethyst}{rgb}{0.6, 0.4, 0.8} 
\definecolor{amber}{rgb}{1.0, 0.75, 0.0} 
\definecolor{burlywood}{rgb}{0.87, 0.72, 0.53} 
\newcommand{\ken}[1]{{\color{amaranth} #1}}
\newcommand{\fab}[1]{{\color{ceruleanblue} #1}}
\newcommand{\change}[1]{{\textcolor{dartmouthgreen}{{#1}}}}

\title{Handle numbers of guts of sutured manifolds and nearly fibered knots.}

\author{Kenneth L. Baker}

\address{Department of Mathematics, University of Miami, Coral Gables, FL 33146, USA}
\email{k.baker@math.miami.edu}

\author{Fabiola Manjarrez-Guti\'errez}
\address{Instituto de Matem\'aticas, Universidad Nacional Aut\'onoma de Mexico, Cuernavaca, Mor., MEXICO}
\email{fabiola.manjarrez@im.unam.mx}

\date{May 2024}

\subjclass[2020]{57K10, 57K35, 57K99}  
\keywords{sutured manifolds, handle number, Heegaard splittings, Morse-Novikov number, tunnel number, nearly fibered knot, guts of Seifert surfaces}

\begin{document}

\begin{abstract}
Extending Haken's Theorem to product annuli and disks for Heegaard splittings of sutured manifolds, we show that the handle number of an irreducible sutured manifold equals the handle number of its guts. 
We further show that reduced sutured manifolds with torus boundary contained in $S^3$ fall in to three types that generalize the three models of guts of knots that are nearly fibered in the instanton or Heegaard Floer sense. In conjunction with these results and another concerning uniqueness of incompressible Seifert surfaces, we show that while many nearly fibered knots have handle number $2$ and a unique incompressible Seifert surface, some have handle number $4$ and others have extra incompressible Seifert surfaces.  Examples of nearly fibered knots with non-isotopic incompressible Seifert surfaces are exhibited.

\end{abstract}

\maketitle

\tableofcontents







\section{Introduction}

The exterior of a Seifert surface $S$ for a link $L=\bdry S$ in a closed oriented $3$-manifold is naturally a sutured manifold $(M_S, \gamma_S)$, called the complementary sutured manifold to $S$.  Its associated reduced sutured manifold, or {\em guts}, is essentially the ``non-product'' part of $(M_S, \gamma_S)$.  We say these are guts of the Seifert surface $S$ too.  In particular, if the guts are empty, then $(M_S,\gamma_S)$ is a product sutured manifold, implying that $L$ is a fibered link with fiber $S$ and  $S$ is the only incompressible Seifert surface for $L$.

The handle number of a sutured manifold (without toroidal sutures) is the number of handles needed in a Heegaard splitting, and the handle number of a knot is the minimum handle number among the complementary sutured manifolds of its Seifert surfaces.
A fibered knot has handle number $0$, so a knot with handle number $2$ may be regarded as close to being fibered.

In the Instanton and Heegaard Floer theories, a knot is fibered if and only if the top grading has rank 1, whence a knot is said to be {\em nearly fibered} if the top grading has rank 2.
It follows from \cite{Juhasz-FloerHomSurDecomp} that nearly fibered knots have unique minimal genus Seifert surfaces.
Building on Baldwin-Sivek \cite{BaldwinSivek}, Li-Ye  \cite{Gutsofnearlyfiberedknots} showed that nearly fibered knots have 
guts of one of the following three models
\begin{itemize}
    \item  (M1) a solid torus with four longitudinal sutures,
    \item  (M2) a solid torus with two sutures each of winding number 2, or
    \item  (M3) a positive Trefoil exterior with two sutures of slope 2.
\end{itemize}
Furthermore each model is realized by some nearly fibered knot. 

Here we show that while many of these Floer nearly fibered knots have handle number $2$ and a unique {\em incompressible} Seifert surface, some have handle number $4$ and others have extra incompressible Seifert surfaces. We say that a link has a sutured manifold as guts if it has a Seifert surface with those guts.

\begin{theorem}\label{thm:stuructreofnearlyfibered}
    Let $K$ be a nearly fibered knot in $S^3$. 
    If $K$ has guts of the first two models then $h(K)=2$,  otherwise $K$ has guts of the third model and $h(K)=4$.
    Moreover, $K$ has an incompressible Seifert surface of non-minimal genus if and only if the guts are of the first model and isolating. 
\end{theorem}

Guts of a sutured manifold are {\em isolating} if they separate off a component of the product piece with a single connected suture; see Definition~\ref{defn:isolation}.

We prove Theorem~\ref{thm:stuructreofnearlyfibered} after discussing some of its ingredients.

\begin{theorem}\label{thm:nonisolatingandincomphpgutstoknothandlenumberanduniqueSeifertsfce}
    Let $F$ be an incompressible Seifert surface for the knot $K$ in $S^3$.  
    If $F$ has guts that are
    connected, incompressibly horizontally prime, and non-isolating, then $F$ is the unique incompressible Seifert surface of $K$ and the handle number of $K$ is the handle number of the guts.
\end{theorem}

The property {\em incompressibly horizontally prime} means that any incompressible oriented surface whose boundary is the sutures is contained in a collar of the boundary; see Definition~\ref{defn:horizontallyprime}.
\begin{proof}
This follows from Theorem~\ref{thm:handlenumberofgutsishandlenumberofmfld} and Corollary~\ref{cor:uniqueSeifert}, a consequence of Theorem~\ref{theo:sutmanwithnonisolatinggutsishoriprime}.
\end{proof}

At its core, this theorem is about pulling  the properties of handle number and incompressibly horizontally prime back to a sutured manifold from its guts. 

Extending Haken's Theorem in Theorem~\ref{thm:haken} to product annuli and product disks for Heegaard splittings of sutured manifolds, we show the handle number of guts always pulls back.
\begin{thm:handlenumberofgutsishandlenumberofmfld}
    Let $(M,\gamma)$ be an irreducible sutured manifold with non-empty guts $(G, \gamma_G)$. 
    Then $h(M,\gamma)= h(G,\gamma_G)$.
\end{thm:handlenumberofgutsishandlenumberofmfld}

Pulling back the property of being incompressibly horizontally prime is more subtle.
\begin{theo:sutmanwithnonisolatinggutsishoriprime}
    Let $(M,\gamma)$ be a sutured manifold  such that each component of $\bdry M$ contains an annular suture. If $(M,\gamma)$ has guts $(G,\gamma_G)$ that are connected, incompressibly horizontally prime, and non-isolating, then $(M,\gamma)$ is incompressibly horizontally prime.
\end{theo:sutmanwithnonisolatinggutsishoriprime}

The hypothesis that the guts are non-isolating is necessary.
Without it, Theorem~\ref{theo:sutmanwithnonisolatinggutsishoriprime} may fail as exhibited in Section~\ref{sec:exampleofnearlyfibredknotwithtypeI}. 
We further explicate the necessity of the guts being connected and incompressibly horizontally prime in  Corollary~\ref{cor:gutsNOTconnORNOTihp}.

Prompted by the observation that nearly fibered knots always have guts whose boundary is a torus,
we determine the nature of guts of incompressible Seifert surfaces whose boundary is a torus.  Generalizing the three models of guts of nearly fibered knots, let us define three types of sutured manifolds with torus boundary:
\begin{itemize}
    \item Type I, a solid torus with four longitudinal sutures,
    \item Type II, a solid torus with two sutures each of winding number $\geq 2$, and    
    \item Type III, an exterior of a non-trivial knot with two sutures whose slope is not the boundary slope of an essential annulus.
\end{itemize}

\begin{theorem}\label{thm:gutstructure}
Suppose the guts $(G, \gamma_G)$ of an incompressible Seifert surface of a knot in $S^3$ has torus boundary.  Then the guts are homeomorphic to a sutured manifold of Type I, Type II, or Type III.
Furthermore:
\begin{itemize}
    \item Type I guts are incompressibly horizontally prime, may be isolating, and have $h(G, \gamma_G) = 2$. If they are isolating, then the product part has two components: $P_1$ meeting just one suture of the guts and $P_3$ meeting the other three sutures. 

    \item Type II  guts  are incompressibly horizontally prime, non-isolating, and have $h(G, \gamma) = 2$.

    \item Type III guts are incompressibly horizontally prime if the slope of $\gamma_G$ is not a boundary slope of $G$. The guts  are non-isolating if the slope of $\gamma_G$ is not the meridional slope of $G$. Their handle number $h(G, \gamma_G)$ is bounded below by twice the Heegaard genus of the Dehn filling of $G$ along the slope of $\gamma_G$.
\end{itemize}

\end{theorem}

Recall that a {\em boundary slope} of a manifold with torus boundary is the slope of the boundary of a properly embedded orientable essential surface.  Here, {\em essential} means that the surface is incompressible, $\bdry$--incompressible, and not $\bdry$--parallel.
As shown by Waldhausen \cite[Lemma 1.10]{waldhausen}, in an irreducible orientable $3$-manifold with toroidal boundary, an orientable incompressible surface is either $\bdry$--incompressible or a $\bdry$--parallel annulus.

\begin{proof}
    Suppose the guts $(G, \gamma_G)$ of the complementary sutured manifold to an incompressible Seifert surface of a knot in $S^3$ has torus boundary.   Since $\bdry G$ is a torus in $S^3$, either $G$ is a solid torus or it is a non-trivial knot exterior.  Any compression of $R(\gamma_G)$ would lead to a compression of the original Seifert surface.  Hence $\gamma_G$ is a collection of annular sutures of the same slope in $\bdry G$.   If $\gamma_G$ has at least $4$ components, then there is a vertical  annulus $A$ that is $\bdry$--parallel to an annulus $A'$ in $\bdry G$ that contains $3$ components of $\gamma_G$.   Unless the annulus $A'' = \bdry G \cut A'$ contains exactly one component of $\gamma_G$ and is isotopic rel-$\bdry$ to $A'$, then $(G,\gamma_G)$ cannot be the guts as it admits a further non-trivial decomposition along $A$.   Of course $A''$ contains exactly one component of $\gamma_G$ when $\gamma_G$ has just four components, and moreover $A''$ is isotopic to $A'$ exactly when $G$ is a solid torus and the sutures $\gamma_G$ have longitudinal slope.  This gives the guts of Type I.  Otherwise $\gamma_G$ has two components. If $G$ is a solid torus, then these components cannot be longitudinal since otherwise $(G,\gamma_G)$ would be a product sutured manifold.  This gives the guts of Type II.  If $G$ is a non-trivial knot exterior, then since $(G,\gamma_G)$ can have no essential vertical annulus, the slope of $\gamma_G$ cannot be the boundary slope of an essential annulus.  This gives the guts of Type III.

    Proposition~\ref{prop:handlesofgeneralguts} addresses the handle numbers of the guts while
    Proposition~\ref{prop:isolationofgeneralguts} addresses when the guts are non-isolating. So now we address here that the guts are incompressibly horizontally prime.  This follows fairly immediately from the classification of types of guts.
    Guts of Types I and II are incompressibly horizontally prime since the only connected incompressible properly embedded orientable surfaces in solid tori are meridional disks, boundary parallel disks, and boundary parallel annuli whose boundary components are essential curves of non-meridional slope (eg. \cite[Lemma 1.10]{waldhausen} or \cite[Lemma 3.7.14]{schultensbook}).  
    If there were guts of Type III that were not incompressibly horizontally prime,  
    then they would have a horizontal surface $F$ that is not isotopic to either $R_+(\gamma_G)$ or $R_-(\gamma_G)$. In particular, $F$ would be a properly embedded, oriented incompressible surface with two boundary components that are the cores of the two sutures $\gamma_G$.  
    However, as $\gamma_G$ is not a boundary slope by hypothesis, $F$ must be a boundary parallel annulus and therefore isotopic to either $R_+(\gamma_G)$ or $R_-(\gamma_G)$ contrary to assumption.
\end{proof}


\begin{prop}\label{prop:seifertfiberedguts}
Let $(G, \gamma_G)$ be guts of Type III.
Suppose $G$ is a non-trivial torus knot exterior where the slope of $\gamma_G$ is not the slope of a regular fiber.  Then $h(G, \gamma_G) = 2$  if each suture meets a regular fiber once, and $4$ otherwise. If the slope of $\gamma_G$ is also not the $0$-slope of the torus knot, then $(G, \gamma_G)$ is incompressibly horizontally prime.
\end{prop}

\begin{proof}
The handle number calculations follow from the application of Lemma~\ref{lem:handletoposn} to torus knot exteriors in Example~\ref{exa:torusknotext}.  Since the only boundary slopes of a torus knot exterior are the slope of a regular fiber of the Seifert fibration and the boundary slope of the Seifert surface, the only incompressible surfaces in a Seifert fibered space, that this sutured manifold is incompressibly horizontally prime follows from Theorem~\ref{thm:gutstructure}.
\end{proof}

\begin{example}
    The $n$-twisted Whitehead double of a $(p,q)$ torus knot will have guts as in Proposition~\ref{prop:seifertfiberedguts} if $n\neq pq$.  In particular, if $n\neq pq \pm1$ then its guts will have handle number $4$ and so will the knot.
\end{example}

\begin{proof}[Proof of Theorem~\ref{thm:stuructreofnearlyfibered}]
A nearly fibered knot has guts of one of three models, as described as above,  \cite{Gutsofnearlyfiberedknots}.  These three models are cases of the three types of guts described in Theorem~\ref{thm:gutstructure} where the third model belongs to the specialization of the Type III guts further discussed in Proposition~\ref{prop:seifertfiberedguts}.

Theorem~\ref{thm:gutstructure}  informs us that the first two models have handle number $2$.  Hence nearly fibered knots with these models of guts must have handle number $2$. Proposition~\ref{prop:seifertfiberedguts} further informs us that the third model has handle number $4$ since a regular fiber of the Seifert fibered structure in the boundary of the exterior of the right handed trefoil has slope $6$.
Since Type III guts are non-isolating by Theorem~\ref{thm:gutstructure}, Theorem~\ref{thm:handlenumberofgutsishandlenumberofmfld} implies that nearly fibered knots with guts of the third model must have handle number $4$.

By Theorem~\ref{thm:gutstructure} all these guts are incompressibly horizontally prime and only Type I guts (which are the first model of guts) may be isolating.
When the guts are non-isolating, Corollary~\ref{cor:uniqueSeifert} shows that they have unique incompressible Seifert surface.  When the guts have Type I and are isolating, we construct an incompressible Seifert surface of higher genus in \S\ref{sec:exampleofnearlyfibredknotwithtypeI}.
\end{proof}

\begin{remark}
Nearly fibered knots are almost-fibered, but almost-fibered knots are not necessarily nearly-fibered.
In \cite{CTPforKnots} an \textit{almost-fibered} knot $K$ is defined as a  non-fibered knot $K$ that possesses a circular Heegaard splitting $(E(K),F,S)$ such that 
\begin{itemize}
    \item $F$ and $S$ are connected Seifert surfaces for $K$ and
    \item $1-\chi(S)$ realizes the circular width of $K$.
\end{itemize}

In \cite{CTPforKnots} is it proved that in a circular thin position the thin levels are incompressible and the thick levels weakly incompressible. An immediate consequence is that knots with a unique incompressible Seifert surface are almost-fibered. 
Also, it follows that knots with a minimal genus Seifert surface of handle number $2$ are almost-fibered.
Thus nearly fibered knots are almost-fibered.

On the other hand, there are almost-fibered knots which are not nearly fibered.  For instance, all non-fibered prime knots with up to ten crossings have minimal genus Seifert surfaces of handle number $2$ (see \cite{Goda-HandleNumberSS}) and hence they are almost-fibered.  However, some of them have non-unique minimal genus Seifert surfaces, and are therefore not nearly fibered. 
\end{remark}

\subsection{Acknowledgements}

KLB was partially supported by the Simons Foundation (gifts \#523883 and \#962034 to Kenneth L.\ Baker).
He also thanks the University of Pisa where a portion of this work was done.

FM-G was partially supported by Grant UNAM-PAPIIT IN113323 and by IMSA, the Institute for the Mathematical Sciences of the Americas at the University of Miami.  

We thank IMSA and the Department of Mathematics at the University of Miami for their support and hospitality during the much of the production of this work.

 \section{Sutured manifolds}

Sutured manifolds were introduced by Gabai \cite{Gabai_foliationsandthetopologyof3manifolds} 
and provide a convenient framework for discussing compression bodies and various generalizations of Heegaard splittings, e.g. \cite{Goda-circlevaluedmorsetheory, Baker-MN}.
Indeed, Casson-Gordon's treatment of Heegaard splittings effectively takes this approach \cite{CG}.

Throughout, for simplicity, we assume our $3$-manifolds are irreducible. 
A {\em sutured manifold} is a compact oriented $3$--manifold $M$ with a disjoint pair of subsurfaces $R_+$ and $R_-$ of $\bdry M$ such that
\begin{itemize}
    \item the orientation of $R_+$ is consistent with the boundary orientation of $\bdry M$ while the orientation of $R_-$ is reversed,
    \item $\bdry M \cut (R_+ \cup R_-)$ is a collection $\gamma$ of annuli $A(\gamma)$ and tori $T(\gamma)$, collectively called the {\em sutures}, and
    \item each annular suture joins a component of $\bdry R_+$ to a component of $\bdry R_-$.
\end{itemize}

Typically, sutured manifolds are written as $(M,\gamma)$ where $\gamma$ represents the sutured structure.  With this sutured structure understood, we also may use $\gamma$ to represent the core curves of $A(\gamma)$, oriented to be isotopic to $\bdry R_+$.  
At times it is convenient to use $\bdry_+ M$ and $\bdry_- M$ to refer to the surfaces $R_+$ and $R_-$.
(One may also care to use $\bdry_v M$, for {\em vertical} boundary, to refer to the annular sutures $A(\gamma)$.)
We may also use variations of $R_\pm(M)$ or $R_\pm(\gamma)$ for these surfaces $R_-$ and $R_+$.

Suppose $(S, \bdry S)$ and $(T,\bdry T)$ are two embedded oriented surfaces in general position in $(M, \bdry M)$. Define $T \dcsum S$ to be the {\em double curve sum of $T$ and $S$}, i.e., cut and paste along the curves of intersection to get an embedded oriented surface representing the cycle $T+S$. The surface $T \dcsum S$ is an embedded oriented surface coinciding with $T \cup S$ outside a regular neighborhood of the intersections $T \cap S$.  See \cite[Definition 1.1]{Scharlemann_SM}.

A {\em decomposing surface} $(S,\bdry S)\subset(M, \bdry M)$ is an oriented properly embedded surface which intersects each toroidal suture in coherently oriented parallel essential circles and each annular suture either in circles parallel to the core of the annulus or in essential arcs (not necessarily oriented coherently). There is a natural sutured manifold structure $(M', \gamma')$ on $M'= M\cut S$, where:
\begin{itemize}
    \item $\gamma'=(\gamma \cap M')\cup \nbhd(S_+\cap R_-)\cup \nbhd(S_- \cap R_+)$,
    \item $R_+'= ((R_+\cap M')\cup S_+)-int(A(\gamma'))$, and
    \item $R_-'= ((R_-\cap M')\cup S_-)-int(A(\gamma'))$.
\end{itemize}
Here $\nbhd(S) = S \times (-\epsilon, \epsilon)$ is a bicollar neighborhood of $S$ where 
$S=S\times \{0\}$ and we set $S_{\pm}= S\times \{\mp \epsilon \}$.  With this arrangement, the normal of the oriented surface  $S_+$ points out of $M'=M\cut S$ and of $S_-$ points into $M'=M\cut S$.
We say $(M',\gamma')$ is obtained from a {\em sutured manifold decomposition} of $(M,\gamma)$ along $S$, and we denote this decomposition as $(M,\gamma) \overset{S}{\leadsto} (M',\gamma')$. 

\subsection{Horizontal surfaces}

\begin{defn}[Cf.{\cite [Defn 9.3]{Juhasz-FloerHomSurDecomp}, \cite[Defn 3.4]{AgolZhang}, \cite[Defn 2.7]{ni-KFHdetectsfibered}}]
\label{defn:horizontal}
Let $(M,\gamma)$ be a sutured manifold without toroidal sutures.  A decomposing surface $S \subset M$ is an {\em incompressible horizontal surface} if 
\begin{enumerate}
    \item $S$ is incompressible,
    \item $\bdry S \subset \gamma$ with $\bdry S$ isotopic to $\bdry R_+(\gamma)$,
    \item $[S, \bdry S] = [R_+(\gamma), \bdry R_+(\gamma)]$ in $H_2(M,\gamma)$, and
    \item no collection of closed components of $S$ is null-homologous.
\end{enumerate}
This definition differs from that of a {\em horizontal surface} in \cite[Defn 9.3]{Juhasz-FloerHomSurDecomp} and \cite[Defn 3.4]{AgolZhang} in several ways.  Notably, an incompressible horizontal surface need not be taut, and we do not require the sutured manifold to be `balanced'.  
\end{defn}

\begin{defn}[{Cf.\ \cite[Defn 9.3]{Juhasz-FloerHomSurDecomp}}]  \label{defn:horizontallyprime}
A sutured manifold $(M,\gamma)$ without toroidal sutures is {\em incompressibly horizontally prime} if any incompressible horizontal surface in $(M,\gamma)$ is isotopic into a collar of  $R_+(\gamma) \cup R_-(\gamma)$. 
\end{defn}

\begin{remark}\label{rem:kobayashi}
    Kobayashi says a sutured manifold $(M,\gamma)$ is an {\em almost product sutured manifold (APSM)} if any incompressible horizontal surface is isotopic either to $R_+(\gamma)$ or to $R_-(\gamma)$ \cite{kobayashi}. Kobayashi shows that the property of being APSM pulls back under product disk decompositions and a set of ``A-operations''.  
    
    Note that $(M,\gamma)$ cannot be APSM if $R_+$ and $R_-$ are both incompressible and $\bdry M$ is disconnected unless $(M,\gamma)$ is a product sutured manifold.  
    However, when $\bdry M$ is connected, an incompressibly horizontally prime sutured manifold $(M,\gamma)$ is an APSM. 
\end{remark}


\subsection{Product decomposition surfaces and guts}

A {\em product annulus} in $(M, \gamma)$ is a properly embedded annulus $A$ in $M$ such that does not cobound a solid cylinder in $M$ and such that one boundary component of $A$ lies on $R_-$ and the other on $R_+$. An essential loop in a product annulus is a {\em horizontal} loop. 
A {\em product disk} is a properly embedded disk $D$ in $M$ such that $\vert D\cap \gamma\vert=2$.  In particular, a product disk $D$ meets $A(\gamma)$ in a pair of spanning arcs of $A(\gamma)$.  A properly embedded arc $\alpha$ in a product disk $D$ is {\em horizontal} if each of the spanning arcs of $D \cap A(\gamma)$ contains an endpoint of $\alpha$.

 A product annulus or product disk is {\em essential} if it is not parallel into $A(\gamma)$.

A decomposing surface is a {\em product decomposition surface} if each of its components a product disk or a product annulus.
It is {\em essential} if each component is essential and no two components are parallel.  It is {\em maximal} if it is not a proper subset of any other essential product decomposition surface.


\begin{defn}[Cf. {\cite[Defn 4.10]{Gabai_foliationsandthetopologyof3manifolds}}]
A sutured manifold $(G, \gamma_G)$ is {\em reduced} if it is not a product sutured manifold and any product disk or product annulus is parallel into $A(\gamma_G)$.
A product decomposition surface of a 
sutured manifold $(M,\gamma)$ is {\em reducing} if it decomposes $(M,\gamma)$ into a product sutured manifold $(P, \gamma_P)$ and a reduced sutured manifold $(G, \gamma_G)$.  
Any reduced sutured manifold $(G, \gamma_G)$ resulting from decomposing $(M,\gamma)$ along a reducing product decomposition surface may be regarded as {\em guts} of $(M,\gamma)$.

As Gabai notes, the guts of a taut sutured manifold  $(M,\gamma)$ are well-defined and independent of choice of maximal essential product decomposition surface. Indeed when $R_\pm(\gamma)$ are only incompressible JSJ Theory \cite{JSJ, Johannson} similarly implies that the guts are well-defined. Hence when $R_\pm(\gamma)$ are incompressible, we may refer to the resulting reduced sutured manifold as the guts of $(M,\gamma)$.

A reduced product decomposition surface is {\em minimal} if it does not properly contain another reducing product decomposition surface.
Observe that any maximal essential product decomposition surface is a reducing product decomposition surface.
Furthermore any minimal reducing product decomposition surface is contained in a maximal essential product decomposition surface and their reduced sutured manifolds are the same.

\end{defn}

\begin{remark}
    Guts of taut sutured manifolds in \cite{AgolZhang} require first decomposing along essential taut horizontal surfaces.  When essential taut horizontal surfaces exist, however, the guts are necessarily disconnected.  In this article, we focus on sutured manifolds with connected guts.  
\end{remark}


\begin{defn}\label{defn:isolation}
    A product annulus $A$ of a minimal reducing product decomposition surface is {\em isolated} if the interior of any path in $(M,\gamma)$ from $A$ to $A(\gamma)$ intersects the guts $(G, \gamma_G)$.  Observe that such an isolated annulus appears in every reducing product decomposition.

    Say the guts $(G, \gamma_G)$ of a sutured manifold $(M, \gamma)$ 
    are {\em non-isolating} if $(M, \gamma)$ has a minimal reducing product decomposition surface without an isolated annulus and {\em isolating} otherwise.
\end{defn}




\subsection{Non-isolating guts means product disk reduction}

\begin{lemma}\label{lem:nonisolatingmeansproductdisks}
    If the guts $(G, \gamma_G)$ of a sutured manifold $(M, \gamma)$ are  non-isolating, then $(M, \gamma)$ has a minimal reducing product decomposition surface consisting of only product disks.
\end{lemma}

\begin{proof}
    Suppose the guts $(G, \gamma_G)$ of $(M, \gamma)$ is non-isolating.  Then is has a minimal reducing product decomposition surface without any isolated annulus.  Among minimal reducing product decomposition surfaces, let $\Pi$ be one with the fewest annuli.
    If $\Pi$ has no annuli, then we are done.  
    So suppose $\Pi$ has an annulus. Then there is path in $(M,\gamma)$ from this annulus of $\Pi$ to $A(\gamma)$ that is disjoint from $(G,\gamma_G)$.  By starting the path at the last annulus of $\Pi$ it hits --- say $A$ --- and then diverting the path along the first product disk it meets to $A(\gamma)$, we may assume this path is also disjoint from $\Pi$ except at its start.  Hence this path lies in a  component of the product sutured manifold $(P,\gamma_P)$ after the decomposition $(M,\gamma)$.  Thus the path is homotopic rel-$\bdry$ into a product disk $D$ of $(P,\gamma_P)$.  Pulling back to $(M,\gamma)$, this product disk $D$ joins the annulus $A$ of $\Pi$ to $A(\gamma)$.
    Then $\bdry \nbhd(A \cup D)$ decomposes as an annulus in each $R_+(\gamma)$ and $R_-(\gamma)$, a parallel copy of $A$, a spanning rectangle of $A(\gamma)$, and a product disk $D'$. See Figure \ref{fig:tradingannulusfordisk}.
    We may now trade the product annulus $A$ for the product disk $D'$ to produce a new minimal reducing product decomposition surface $\Pi' = (\Pi- A) \cup D$.  However $\Pi'$ has fewer annuli than $\Pi$, contrary to assumption.
    \end{proof}

\begin{remark}
    In fact one can show that if  $(M, \gamma)$ has  non-isolating guts $(G, \gamma_G)$, then any  reducing product decomposition surface has no isolated product annulus.
\end{remark}

\subsection{Isolation of guts with torus boundary}

\begin{prop}\label{prop:isolationofgeneralguts}
Suppose the guts $(G, \gamma_G)$ of the complementary sutured manifold $(M,\gamma)$ to an incompressible Seifert surface $S$ of a knot in $S^3$ is homeomorphic 
to a sutured manifold of Type I, Type II, or Type III. 
Then
\begin{itemize}
    \item Type I guts are either non-isolating or any minimal reducing product decomposition surface of $(M,\gamma)$ has exactly one isolated product annulus,

    \item Type II guts are non-isolating, and

    \item Type III guts are non-isolating if the slope of $\gamma_G$ is not the meridional slope of $G$. 
\end{itemize}

\end{prop}

\begin{remark}
Proposition~\ref{prop:isolationofgeneralguts} may be viewed as clarifying \cite[Remark~2.1(2)]{Gutsofnearlyfiberedknots}.
\end{remark}

\begin{proof}
Let $\Pi$ be a minimal reducing product decomposition surface of $(M,\gamma)$ giving the 
 decomposition
 \[(M,\gamma) \overset{\Pi}{\leadsto} (P, \gamma_P) \sqcup (G, \gamma_G)\]
 with product sutured manifolds  $(P, \gamma_P)$ and guts $(G, \gamma_G)$. 
Let $(P,\gamma_P) = (F \times [-1,1], \bdry F \times [-1,1])$ where we identify $F$ with $F \times \{0\}$. 

Assume that the guts are isolating so that there is at least one isolated annulus of $\Pi$.

If two components of $\gamma_G$ both result from the same isolated annulus of $\Pi$, then the guts must be Type I (because $\gamma \neq \emptyset$).  Since  $R(\gamma)$ has no closed components, the two components of $\gamma_G$ cannot be adjacent.  Then the annulus in $G$ joining the cores of these two components would then give rise to an embedded Klein bottle in $(M,\gamma)$.  This cannot occur since $M$ is a submanifold of $S^3$.
Therefore every annulus of $\Pi$ must join a suture of $\gamma_G$ with a suture of $\gamma_P$.
In particular, the components of $\gamma_G - \gamma$ are identified with the annuli of $\Pi$.


If $(G,\gamma_G)$ is of Type II or III, then $\gamma_G=\{\gamma_0,\gamma_1\}$.
Since the guts are isolating, then say $\gamma_1$ corresponds to an isolated annulus $\Pi_1$ of $\Pi$.   Let $F_1$ be the component of $F$ that meets $\Pi_1$.  Note that $\gamma_0$ cannot also correspond to an isolated annulus of $\Pi$ since $\gamma \neq \emptyset$.
 Hence the boundary of $F_1$ is the core curve of $\gamma_1$, implying that the slope of $\gamma_1$ is null-homologous in $S^3 \cut G$.  If $(G,\gamma_G)$ is of Type II,  then $\gamma_1$ must be a longitude for the solid torus $G$, which is contrary to the sutures having winding number of at least $2$ about $G$. If $(G,\gamma_G)$ is of Type III, then $\gamma_1$ must be a meridian of $G$ as claimed.


If $(G,\gamma_G)$ is of Type I, then $\gamma_G= \{\gamma_0, \gamma_1, \gamma_2, \gamma_3\}$ where these sutures are numbered in cyclic order around $\bdry G$.
Let $\Pi_i$ be the annulus of $\Pi$ that corresponds to $\gamma_i \in \gamma_G - \gamma$.
Set $\mathcal{I}\subset \{0,1,2,3 \}$, and denote by $F_\mathcal{I}$ the component of $F$ that meets the annuli $\Pi_i$ for $i\in \mathcal{I}$ and is disjoint from $\Pi_j$ for $j \not \in \mathcal{I}$.  

We may assume $\gamma_0$ does not correspond to an isolated annulus of $\Pi$.  In particular, either $\gamma = \gamma_0$ or $\gamma$ is a suture of the component $F_\mathcal{I} \times I$ of $P$ for which $0 \in \mathcal{I}$.  Consequently, if $0 \not \in \mathcal{I}$, then $\bdry F_{\mathcal{I}} \times [-1,1] = \bigcup_{i \in \mathcal{I}} \Pi_i= \bigcup_{i \in \mathcal{I}} \gamma_i$.


\begin{claim}
$F$ cannot have the components $F_2$ and $F_1$, $F_2$ and $F_3$, $F_{01}$, $F_{12}$, $F_{23}$,  $F_{03}$, $F_{13}$, $F_{123}$, $F_1$ and $F_3$, $F_{02}$.
\end{claim}
\begin{proof}
Observe that if $F$ contains components  $F_2$ and either $F_1$ or $F_3$, then one observes that either $R_+(\gamma)$ or $R_-(\gamma)$, and hence $S$, has a closed component. This contradicts that $S$ is a Seifert surface.  Similarly  $F$ cannot have a component $F_{01}$, $F_{12}$, $F_{23}$, or $F_{03}$, since otherwise $S$ would have a closed component.

Let $A_{ij}$ be a properly embedded annulus in $G$ joining essential loops in $\gamma_i$ and $\gamma_j$.
If $F$ were to have a component $F_{13}$, then the union of $F_{13} \times \{0\}$ with $A_{13}$ would produce a closed non-orientable surface embedded in $S^3$, a contradiction.
Similarly, if $F$ were to have a component $F_{123}$, then the union of $F_{123} \times \{-1, +1\}$ with annuli $A_{12}$, $A_{23}$, and $A_{13}$ would produce a closed non-orientable surface embedded in $S^3$, a contradiction.

If $F$ were to have components $F_1$ and $F_3$, then the union $F_1\times \{0\} \cup A_{13} \cup F_3 \times \{0\}$ would be an embedded closed surface $\Sigma$ in $S^3$.  Since we have already seen that $F$ cannot also have $F_2$ as a component, it must have $F_{02}$ as a component.  But then a path in $F_{02} \times \{0\}$ from $\gamma_0$ to $\gamma_2$ would complete with an arc in $G$ to a loop that transversally intersects $\Sigma$ just once.  Hence $\Sigma$ would be a non-separating surface in $S^3$, a contradiction.

Finally, if $F$ had the component $F_{02}$, then it must either have components $F_1$ and $F_3$ or the component $F_{13}$. But these have been ruled out.

\end{proof}

From this claim, it then follows that $\vert F \vert \leq 2$. If $\vert F \vert = 1$, then $F = F_{0123}$ and the guts are non-isolating. If $\vert F \vert =2$, 
then $F = F_i \sqcup F_{0jk}$ for $\{i,j,k\} = \{1,2,3\}$.
These cases are all equivalent up to renumbering since $F_{0jk}$ meets $\gamma$, and in each case $\Pi$ has exactly one isolated product annulus.
\end{proof}

\section{Handle numbers}


For definitions and details about Heegaard splittings and handle numbers of sutured manifolds see \cite{BMG}.  We briefly overview them here.

Let $(M,\gamma)$ be a sutured manifold without toroidal sutures.  In what follows, we may suppress the sutures from the notation. Recall that compression bodies may be viewed as sutured manifolds.
A {\em  Heegaard splitting} of $M$, denoted $(M,\Sigma)$ or $(A,B;\Sigma)$, is a decomposition of $M$ into a compression body $A$ and dual compression body $B$ along a surface $\Sigma = R_+(A) = R_-(B)$.  (In these compression bodies, $R_-(A)$ and $R_+(B)$ are incompressible or empty.) 
Every sutured manifold $M$ without toroidal sutures admits a Heegaard splitting. 
The {\em handle number} of a Heegaard splitting $(M,\Sigma) = (A,B;\Sigma)$ of a sutured manifold $M$ is $h(M,\Sigma) = h(A)+h(B)$.
The {\em handle number} 
$h(M,\gamma)$ of the sutured manifold $M$ is the minimum of $h(M,\Sigma)$ among all Heegaard splittings of $(M,\gamma)$.
Observe that this is the minimum number of $0$--, $1$--, $2$--, and $3$--handles needed to construct $M$ from a collar of $R_-(M)$.  
In the case of a product sutured manifold $M = F\times [-1,1]$ where $R_{\pm}(M) = F \times \pm1$, the Heegaard surface $F \times 0$ gives a Heegaard splitting of $M$ for which the compression bodies $A$ and $B$ are trivial. Thus its  handle number is zero.  Indeed, the handle number of a sutured manifold is $0$ exactly when the sutured manifold is a  product.

\subsection{Handle number of the guts pulls back}

It turns out that Haken's Theorem about Heegaard splittings and reducing spheres \cite{Haken} and its extension to boundary-reducing disks \cite{CG,BO} generalizes to product disks and annuli in sutured manifolds.

\begin{theorem}\label{thm:haken}


    Let $Q$ be an incompressible product annulus or product disk in a sutured manifold $(M,\gamma)$. Then for any Heegaard surface $\Sigma$ of $(M,\gamma)$,  there is a product annulus or disk $Q'$ with $\bdry Q = \bdry Q'$ so that $Q' \cap \Sigma$ is a single horizontal curve in $Q'$.  If $(M,\gamma)$ is irreducible, then $Q$ is isotopic rel-$\bdry$ to $Q'$. 
\end{theorem}

Our proof follows Jaco's proof of Haken's Theorem \cite[II.7]{jaco} closely.  We use his {\em Type-A} isotopies presented there.  However we need a slight specialization of his hierarchy which we now review.

Recall that a {\em hierarchy} for a compact surface $T$ is a sequence of pairs 
\[ (T_0, \alpha_0), \dots, (T_n, \alpha_n)\]
where $T=T_0$, each $\alpha_i$ is an essential properly embedded arc or loop in $T_i$,  $T_{i+1} = T_i \cut \alpha_i$, and $T_{n+1}$ is a union of disks.  Jaco proves the following combinatorial lemma.

\begin{lemma}[{\cite[II.9.Lemmma]{jaco}}]
\label{lem:JacoLemma}
Let $T$ be a planar surface with $b$ boundary components and $z\geq 0$ components that are not disks.
Let $(T_0, \alpha_0), \dots, (T_n, \alpha_n)$ be any hierarchy in which each $\alpha_i$ is an arc.  If $T_{n+1}$ is a union of $d$ disks, then $d \leq b-z$.  \qed
\end{lemma}

Let us slightly modify this for our purposes.  For a compact surface $T$ with a specified boundary component $C$, a {\em $C$--protected hierarchy} for $T$ is a hierarchy $(T_0, \alpha_0), \dots, (T_n, \alpha_n)$ in which only $\alpha_n$ meets $C$.  In particular $T_n$ is a union of disks and one annulus, $C$ is a boundary component of that annulus, and $\alpha_n$ is a spanning arc of that annulus (and hence $\alpha_n$ is essential in the annulus). Observe then, that $|T_n|=|T_{n+1}|$ since $\alpha_n$ is non-separating. 

\begin{proof}
We prove this theorem explicitly only in the case that $Q$ is an incompressible product annulus.  The case for product disks is even simpler and closer to Jaco's proof of Haken's Lemma. 
Goda proves this for product disks (and gives a generalization for disks crossing the sutures more times) in \cite[Proposition 3.1]{goda-murasugi} when $M$ is irreducible.

Let $\Sigma$ be a Heegaard surface dividing $(M,\gamma)$ into compression bodies $A,B$.  
Among properly embedded annuli with the same boundary as $Q$, assume $Q$ is one that minimizes  $|Q \cap \Sigma|$.  If $M$ is irreducible, instead isotop the annulus $Q$ to minimize $|Q \cap \Sigma|$.  

{\bf Claim 1.}  Both $Q \cap A$ and $Q \cap B$ are incompressible in $A$ and $B$ respectively.  

Suppose to the contrary that, say, $Q \cap A$ admits a compressing disk $\delta_A$ in $A$.   Since $Q$ is incompressible, $\bdry \delta_A$ must bound a disk $\delta_Q$ in $Q$.  Yet since $\delta_A$ is a compressing disk for $Q \cap A$, $\bdry \delta_A$ does not bound a disk in $Q \cap A$. Hence $\delta_Q$ must intersect $\Sigma$ in its interior.  
Then $Q'=(Q-\delta_Q) \cup \delta_A$ is then an incompressible product annulus with the same boundary as $Q$, and a slight isotopy of this annulus off $\delta_A \subset A$ has fewer intersections with $\Sigma$ than $Q$.   If $M$ is irreducible,  then the sphere  $\delta_A \cup \delta_Q$ bounds a ball and $\delta_Q$ may be isotoped to $\delta_A$ through this ball and hence  $Q'$ is actually isotopic to $Q$.

{\bf Claim 2.}   $Q$ intersects $\Sigma$ in a single curve.

Assume for contradiction that $|Q \cap \Sigma| > 1$.

Suppose $Q \cap A$ is not a union of disks and one annulus.  Set $T=Q \cap A$. As it is a subsurface of $Q$, $T$ must also be planar. Furthermore, observe that $|Q \cap \Sigma| = |\bdry T|-1$. 

Among complete systems of meridional disks for $A$, let $\Delta_A$ be one that minimizes $|\Delta_A \cap T|$.  Note that $\bdry \Delta_A \cap (Q \cap \bdry_-A)=\emptyset$.
Since $T$ is incompressible by Claim 1, we may assume $T \cap \Delta_A$ is a union of arcs.  By the assumed minimality of $|\Delta_A \cap T|$, the arcs are essential in $T$.

 These arcs define a partial hierarchy $(T_0, \alpha_0), \dots, (T_k, \alpha_k)$ of $T=T_0$ where each $\alpha_i$ bounds a subdisk $\delta_i$ of $\Delta_A$ that is disjoint from $\alpha_j$ for $j>i$.
Successively performing Type-A isotopies on $Q$ along the subdisks $\delta_i$ isotops the annulus $Q=Q_0$ through a sequence of surfaces $Q_i$ with $Q_i \cap A = T_i$ to a surface $Q_{k+1}$ with $Q_{k+1} \cap A = T_{k+1}$ that is disjoint from $\Delta_A$.

Since $A \cut \Delta_A$ is a product $\bdry_- A \times I$ and $T_{k+1}$ is an incompressible surface properly embedded in $A \cut \Delta_A$ meeting $\bdry_-A$ in a single curve, $T_{k+1}$ is a union of a single spanning annulus $Q_A$ and a collection of planar surfaces that are $\bdry$--parallel into $\bdry_-A \times \{+1\}$.
These $\bdry$--parallelisms induce a sequence of Type-A isotopies along arcs $\alpha_{k+1}, \dots, \alpha_{n-1}$  that reduce the collection of planar surfaces to parallel copies of images components of $\Delta_A$ in $A \cut \Delta_A$. These Type-A isotopies continue the isotopy of $Q$ from $Q_{k+1}$ to $Q_{n}$ rendering $T_{k+1}$ into $T_n = Q_n \cap A$, the union of the annulus $Q_A$ and a collection of disks. Observe that $|T_n|=|Q_n \cap \Sigma|$.

Letting $\alpha_n$ be a spanning arc of the annulus $Q_A$, this defines a $C$--protected hierarchy for $T$ where $C = Q \cap \bdry_-A$.  Since $T$ is a planar surface and the hierarchy uses only arcs, Lemma~\ref{lem:JacoLemma} implies that $|T_n| \leq |\bdry T|-z$ where $z$ is the number of components of $T$ that are not disks.  By assumption, $z\geq 2$.  Thus 
\[|Q_n \cap \Sigma| = |T_n| \leq |\bdry T|-z \leq |\bdry T|-2 = |Q \cap \Sigma|-1\]
which contradicts the presumed minimality of $|Q \cap \Sigma|$ since $Q$ is isotopic to $Q_n$.

Thus $Q \cap A$ must be a union of a single annulus and some number of disks.
As the same argument also shows that $Q \cap B$ must be a union of a single annulus and some number of disks,
we conclude that in fact $Q \cap A$ and $Q \cap B$ are each single annuli since $Q$ is an annulus itself. This establishes Claim 2 and proves the theorem.  
\end{proof}

\begin{theorem}\label{thm:handlenumberofgutsishandlenumberofmfld}
    Let $(M,\gamma)$ be an irreducible sutured manifold with non-empty guts $(G, \gamma_G)$. Then $h(M,\gamma)= h(G,\gamma_G)$. 

    
\end{theorem}

\begin{proof}
Let $\Pi$ be a minimal reducing product decomposition surface for $(M, \gamma)$ giving the decomposition
\[(M,\gamma) \overset{\Pi}{\leadsto} (P, \gamma_P) \cup (G, \gamma_G)\]
into a product sutured manifold $(P, \gamma_P)$ and the guts $(G, \gamma_G)$.
Since $(P,\gamma_P)$ is a product manifold, $h(P,\gamma_P)=0$. Thus $h(M,\gamma) \leq h(P, \gamma_P) + h(G,\gamma_G) = h(G,\gamma_G)$ where the inequality is due to \cite[Lemma 2.8]{BMG}.

On the other hand if $\Sigma$ is a Heegaard surface for $(M,\gamma)$ realizing $h(M,\gamma)$.  Then Theorem~\ref{thm:haken} implies that $\Pi$ may be isotoped so that each component of $\Pi$ intersects $\Sigma$ in a single horizontal curve.  Hence $\Sigma$ restricts to a Heegaard surface for $(P, \gamma_P)$ and $G, \gamma_G)$.  Therefore $h(M,\gamma) \geq h(P,\gamma_P) + h(G,\gamma_G) = h(G,\gamma_G)$.

Hence it follows that  $h(M,\gamma) = h(G,\gamma_G)$.
\end{proof}

\begin{figure}
    \centering
    \includegraphics[width=.9\textwidth]{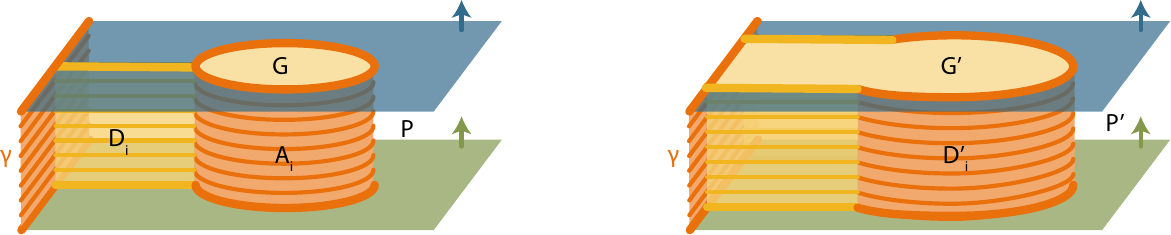}   
    \caption{Using a product disk $D_i$, a non-isolating product annulus $A_i$ may be traded for a product disk $D_i'$.}
    \label{fig:tradingannulusfordisk}
\end{figure}


\subsection{On sutured manifolds derived from bridge positions} 
\label{sec:minimizinglinhandlenumbers}
Given a compact connected oriented $3$-manifold $M$ where $\bdry M$ is a single torus, let $(M,\gamma_{\sigma,2n})$ be the sutured manifold structure consisting of $2n$ sutures of slope $\sigma$ in $\bdry M$ where $n$ is a positive integer.
Let $K_\sigma$ denote a knot in a closed $3$-manifold $M_\sigma$ with exterior $M$ and meridian $\sigma$ so that $M_\sigma$ is the Dehn filling of $M$ along the slope $\sigma$ and $K_\sigma$ is the core of that filling. Recall that a $(g,b)$--position of $K_\sigma$ is a division of $K_\sigma$ by a genus $g$ Heegaard surface of $M_\sigma$ into $b$ mutually boundary-parallel properly embedded arcs in each of the two handlebodies of the Heegaard splitting, \cite{Dollgbposition}.

\begin{lemma}\label{lem:easygbbbounds}
    A $(g,b)$--position of $K_\sigma$ in $M_\sigma$ induces the sutured manifold structure $(M,\gamma_{\sigma,2b})$ with handle number 
    \[ 2(b-1) \leq h(M,\gamma_{\sigma,2b}) \leq 2(g+b-1).\]
    In particular, if  $K_\sigma$ has a $(0,b)$--position (so that $M_\sigma=S^3$), then $h(M,\gamma_{\sigma,2b}) = 2(b-1)$. 
\end{lemma}

\begin{proof}
    Let $\hat{S}$ be genus $g$ Heegaard surface for $M_\sigma$ giving a $(g,b)$--position of $K_\sigma$.  Thus $\hat{S}$ splits $M_\sigma$ into handlebodies $\hat{A}$ and $\hat{B}$ for which $\hat{A} \cap K_\sigma$ and $\hat{B} \cap K_\sigma$ are each $b$ boundary parallel arcs.   Let $S = M \cap \hat{S}$, $A=M \cap \hat{A}$, and $B = M \cap \hat{B}$. Then $\bdry S$ is a collection of $2b$ curves of slope $\sigma$ in $\bdry M$ that alternate in orientation, inducing the sutured manifold structure $(M, \gamma_{\sigma, 2b})$.  Furthermore $(S;A,B)$ is a Heegaard splitting for this sutured manifold.  Since $\hat{S}$  compresses along $g$ disks in $\hat{A}$ to yield a ball containing the boundary parallel arcs $\hat{A} \cap K_\sigma$, a further $b-1$ compressions yield $b$ balls each containing a single unknotted arc of $\hat{A} \cap K_\sigma$. Thus the compression body $A$ has handle number $h(A)=g+b-1$.  Similarly $h(B)=g+b-1$ so that $h(S;A,B)=2(g+b-1)$.  Therefore we have the upper bound $h(M, \gamma_{\sigma, 2b}) \leq 2(g+b-1)$.

    Since the $b$ annuli of each $R_+(\gamma_{\sigma, 2b})$ and $R_-(\gamma_{\sigma, 2b})$ must be joined together for a Heegaard splitting of $(M, \gamma_{\sigma, 2b})$, we have the lower bound 
    $2(b-1) \leq h(M, \gamma_{\sigma, 2b})$.
\end{proof}

\begin{lemma}\label{lem:handletoposn}
Consider the sutured manifold $(M,\gamma_{\sigma,2n})$ and knot $K_\sigma$ in the Dehn filling $M_\sigma$ as above.
If $h=h(M,\gamma_{\sigma,2n})$, then the Heegaard genus of $M_\sigma$ is at most $h/2+1-n$ and $K_\sigma$ has a $(h/2+1-n, n)$--position in $M_\sigma$.   

In particular, if $n=1$  then  $g(M_\sigma) \leq h/2$ and $K_\sigma$ has a $(h/2,1)$--position in $M_\sigma$.
\end{lemma}

\begin{proof} 
If $h=h(M,\gamma_{\sigma,2n})$, then $(M,\gamma_{\sigma,2n})$ has a Heegaard splitting $(S;A,B)$ in which $A$ and $B$ are connected compression bodies with $h(A) = h(B) = h/2$, $S = \bdry_+ A = \bdry_- B$, and each $R_+ = \bdry_+ B$ and $R_- = \bdry_- A$ consist of $n$ annuli. Observe this also means that each $A$ and $B$ have $2n$ annular sutures. As $n-1$ $1$--handles of $A$ are required to connect thickenings of the $n$ annuli $R_-=\bdry_- A$, the remaining contribute to the genus of $S = \bdry_+ A$.  Hence $g(S) = h/2 + 1-n$.  
Furthermore, observing that the sutured manifold of $n$ thickened annuli connected by $n-1$ $1$--handles is homeomorphic to the exterior of the trivial $n$--strand tangle in the ball, attaching another $h/2 + 1-n$ $1$--handles yields a sutured manifold homeomorphic to the exterior of $n$ mutually $\bdry$--parallel arcs in a genus $h/2+1-n$ handlebody.  Thus Dehn filling $\bdry M$ along the slope $\sigma$ of $\gamma$ induces a splitting $(\hat{S};\hat{A},\hat{B})$ of $M_\sigma$ in which $\hat{S}$ is the surface $S$ capped off with $2n$ disks and $\hat{A}$ and $\hat{B}$ are the handlebodies obtained by attaching thickened disks to $\bdry A \cut S$ and $\bdry B \cut S$.  The cores of these thickened disks join together to form $K_\sigma$.  Hence this splitting demonstrates that the Heegaard genus of $M_\sigma$ is at most $h/2+1-n$ and $K_\sigma$ has a $(h/2+1-n, n)$--position in $M_\sigma$.
\end{proof}

\begin{example}\label{exa:solidtorusmeridian}
Let $M$ be the solid torus with meridian $\mu$.  When $\Delta(\sigma,\mu)=1$, the sutures are longitudes so $(M,\gamma_{\sigma, 2})$ is a product, so $h(M,\gamma_{\sigma, 2})=0$.
When $\Delta(\sigma,\mu)\neq 1$, $M_\sigma$ is a non-trivial lens space so $g(M_\sigma)=1$. So Lemma~\ref{lem:handletoposn} says that $h(M,\gamma_{\sigma, 2}) \geq 2$.  However, as $K_\sigma$ is a core of one of the Heegaard solid tori, it has a $(1,1)$--position in the lens sapce.  Thus Lemma~\ref{lem:easygbbbounds} implies that $h(M,\gamma_{\sigma, 2}) \leq 2$. 

Then, we have
 \[h(M,\gamma_{\sigma, 2}) = \begin{cases}
 0 & \mbox{ if } \Delta(\sigma,\mu)=1 \\ 
 2 & \mbox{ otherwise}.
 \end{cases}
 \]
\end{example}

\begin{example}\label{exa:torusknotext}
Let $M$ be the exterior of a non-trivial torus knot. 
Since a torus knot has tunnel number $1$, $K_\sigma$ has a $(2,1)$--position in $M_\sigma$ for any slope $\sigma$.  So Lemma~\ref{lem:easygbbbounds} says that $h(M,\gamma_{\sigma, 2}) \leq 4$. 

Let $\rho$ be the slope of a Seifert fiber.   Then, by the classification of Dehn surgeries on torus knots \cite{moser}, $M_\sigma$ is a lens space if $\Delta(\sigma,\rho) = 1$ and $M_\sigma$ is either a small Seifert fibered space or a connected sum of lens spaces if $\Delta(\sigma,\rho)\neq 1$.  In the latter case, $g(M_\sigma)=2$ and Lemma~\ref{lem:handletoposn} says that $h(M,\gamma_{\sigma, 2}) \geq 4$ so that $h(M,\gamma_{\sigma, 2}) = 4$.  However in the former case $g(M_\sigma)=1$ and $K_\sigma$ has a $(1,1)$--position.  So  Lemma~\ref{lem:handletoposn} and Lemma~\ref{lem:easygbbbounds} imply that $h(M,\gamma_{\sigma, 2})=2$.

 In summary, we have
 \[h(M,\gamma_{\sigma, 2}) = \begin{cases}
 2 & \mbox{ if } \Delta(\sigma,\rho)=1 \\ 
 4 & \mbox{ otherwise}.
 \end{cases}
 \]
 \end{example}

Observe that if a knot has a $(g,b)$--position with $b\geq 2$, then a bridge stabilization of the knot gives a $(g,b+1)$--position while a meridional stabilization of the Heegaard surface (i.e.\ a Heegaard stabilization along a bridge arc of $K_\sigma$) gives a $(g+1,b-1)$--position. See, e.g.\ \cite{MMS-gb,saito-gb}.
 
\begin{lemma}\label{lem:parallelsutures}
$h(M,\gamma_{\sigma,2n}) \leq h(M,\gamma_{\sigma,2(n+1)}) \leq h(M,\gamma_{\sigma,2n}) +2$
\end{lemma}
\begin{proof}
Let $h_n = h(M,\gamma_{\sigma,2n})$ for positive integers $n$.
Lemma~\ref{lem:handletoposn} shows that $(M,\gamma_{\sigma,2n})$ is realized by a $(h_n/2+1-n,n)$--position of the associated knot $K_\sigma$.
A bridge stabilization of this position gives a  $(h_n/2+1-n,n+1)$--position of $K_\sigma$.  
Hence  Lemma~\ref{lem:easygbbbounds} implies
\[h_{n+1} \leq 2((h_n/2+1-n) + (n+1) -1) = h_n+2.\]

On the other hand, Lemma~\ref{lem:handletoposn} shows that $h(M,\gamma_{\sigma,2(n+1)})$ is realized by a $(h_{n+1}/2-n,n+1)$--position of the associated knot $K_\sigma$.  A meridional stabilization along an arc of $K_\sigma$  yields  a $(h_{n+1}/2-n+1,n)$--position of $K_\sigma$.   
Hence  Lemma~\ref{lem:easygbbbounds}  implies
\[h_n \leq 2((h_{n+1}/2-n+1) + n -1) = h_{n+1}.\]
\end{proof}

\subsection{Handle numbers of sutured manifolds of Type I, II, III}

\begin{prop}\label{prop:handlesofgeneralguts}

Let $(G, \gamma_G)$ be homeomorphic 
to a sutured manifold of Type I, Type II, or Type III.
Then:
\begin{itemize}
    \item Type I has $h(G, \gamma_G) = 2$, 
    \item Type II  has $h(G, \gamma_G) = 2$, and
    \item Type III has $h(G, \gamma_G)$  bounded below by twice the Heegaard genus of the Dehn filling of $G$ along the slope of $\gamma_G$.
\end{itemize}

In particular, if $(G, \gamma_G)$ is a non-trivial torus knot exterior where the slope $\sigma$ of $\gamma_G$ is not the slope $\rho$ of a regular fiber, then 
$h(G, \gamma_G) = 2$ if $\Delta(\rho,\sigma)=1$ and $h(G, \gamma_G) = 4$ otherwise.
\end{prop}

\begin{proof}
For Type I sutured manifolds $(G,\gamma_G)$ we have that $\gamma_G= \gamma_{\sigma, 4}$ where $\sigma$ is a longitude.  
Since $h(G,\gamma_{\sigma,2n}) =0$ if and only if $(G,\gamma_{\sigma,2n})$ is a product sutured manifold, $h(G,\gamma_{\sigma,2n}) =0$ if and only if $n=1$ (because a connected product sutured manifold has connected $R_+$). 
Then, by Lemma~\ref{lem:parallelsutures}, we have
\[  0 = h(G,\gamma_{\sigma,2}) \leq h(G,\gamma_{\sigma,4}) \leq h(G,\gamma_{\sigma,2}) + 2  = 2 \]
where the first inequality is actually strict.
Thus $h(G,\gamma_{\sigma,4}) =2$.  

For Type II sutured manifolds $(G,\gamma_G)$ we have that $\gamma_G=\gamma_{\sigma, 2}$ where $\sigma$ intersects the meridian at least twice. Example~\ref{exa:solidtorusmeridian} then shows that
$h(G, \gamma_{\sigma,2}) = 2$.

 For Type III sutured manifolds $(G,\gamma_G)$  we have that $\gamma_G= \gamma_{\sigma,2}$ with $\sigma$ different from the slope of an essential annulus.  Lemma~\ref{lem:handletoposn} implies that the Heegaard genus of the Dehn filling of $G$ along $\sigma$ is at most $h(G,\gamma_{\sigma, 2})/2$.

When $G$ is a non-trivial knot exterior, the handle number calculation follows from Example~\ref{exa:torusknotext}.
\end{proof}

\subsection{Comparisons of handle number with tunnel number} 

With the bounds established in Section~\ref{sec:minimizinglinhandlenumbers}, we may make comparisons of handle numbers of a fixed $3$-manifold with torus boundary for different suture structures.




As before, let $M$ a $3$--manifold where $\bdry M$ is a torus, let $M_\sigma$ be its Dehn filling along the slope $\sigma$, and let $K_\sigma$ be the knot that is the core curve of this filling. 
Let $(M,\gamma_+)$ be the sutured manifold structure with $R_+(\gamma_+) = \bdry M$.  Thus $h(M,\gamma_+) = 2\TN(M) +  2$ where $\TN(M)$ is the tunnel number of $M$.  Note that $\TN(M) = \TN(K_\sigma)$ for any slope $\sigma$.

\begin{example}\label{ex:tn}
For any slope $\sigma$ in $\bdry M$, $g(M_\sigma) \leq \TN(K_\sigma)+1$ and $K_\sigma$ has a $(\TN(K_\sigma)+1,1)$--position in $M_\sigma$.  Thus Lemma~\ref{lem:easygbbbounds} implies that $h(M, \gamma_{\sigma,2}) \leq 2(\TN(K_\sigma)+1) = h(M,\gamma_+)$.
\end{example}

\begin{lemma}\label{lem:gbpositionsandtunnelnumber}
For each positive integer $n$, $h(M,\gamma_+) \leq h(M,\gamma_{\sigma,2n})+2$.

Moreover, $h(M,\gamma_+) = h(M,\gamma_{\sigma,2n})+2$ if and only if $K_\sigma$ has a $(\TN(K_\sigma)-n+1,n)$--position.
\end{lemma}

\begin{proof}
Let $h_n = h(M, \gamma_{\sigma,2n})$ for positive integers $n$, and let $h_+ = h(M,\gamma_+)$. 
By Lemma~\ref{lem:handletoposn}, $K_\sigma$ has a $(g,n)$--position with $g=h_n/2+1-n$.  
Since a $(g,n)$--position of $K$ gives the bound $\TN(K) \leq g+n-1$, we have  $h_+ = 2\TN(K) + 2 \leq 2(g+n-1)+2 = h_n +2$.

Suppose $h_+ = h_n+2$. Lemma~\ref{lem:handletoposn} implies that $K$ has a $(g,n)$--position with $g=h_n/2+1-b$. But $h_n= h_+-2=2\TN(K)$, so $g=\TN(K)-n+1$. 

Now suppose $K$ has a $(g,n)$--position with $\TN(K) = g+n-1$. 
This induces the bound $h_n \leq 2\TN(K)$.
Since $2\TN(K)+2=h_+$, 
we then obtain $h_n +2 \leq h_+$.
However, since  $h_+ \leq h_n +2$ for all $n \geq 1$, we must have $h_n +2 \leq h_+$.
\end{proof}

\begin{cor}
Let $h_+ = h(M,\gamma_+) = 2\TN(M)+2$ and let $h_1 = h(M,\gamma_{\sigma,2})$.  Then $h_1 \leq h_+ \leq h_1+2$.

Moreover $h_+ = h_1+2$ if and only if $K_\sigma$ has a $(\TN(M),1)$--position in $M_\sigma$.
\end{cor}

\begin{proof}
    This follows from Example~\ref{ex:tn} and Lemma~\ref{lem:gbpositionsandtunnelnumber}.
\end{proof}



\begin{example}

Let $M$ be the exterior of a knot $K$ in $S^3$ with meridian $\mu$.  As before, set $h_+ = h(M,\gamma_+) = 2\TN(K)+2$ and $h_n = h(M,\gamma_{\mu,2n})$ for positive integers $n$.

\begin{itemize}
    \item 
The unknot is the only one-bridge knot.  It has tunnel number $0$. For this we have 
\[ h_+ = 2 \quad \mbox{and} \quad
h_n =  2(n-1) \quad n \geq 1. 
\]

\item
Two-bridge knots all have tunnel number $1$.  Thus we have
\[ h_+ = 4 \quad \mbox{and} \quad
h_n = \begin{cases} 2 & n=1,2\\ 2(n-1) & n \geq 2. \end{cases}
\]

\item
Three-bridge knots have tunnel number at most $2$.  
\begin{itemize}
    \item 
    A three-bridge knot with tunnel number $1$ has a $(1,1)$--position \cite{Tunnel1Bridge3}. Thus a three-bridge knot with tunnel number $1$ has
\[ h_+ = 4 \quad \mbox{and} \quad
h_n = \begin{cases} 2 & n=1\\ 4 & n=2,3 \\ 2(n-1) & n \geq 3. \end{cases}
\]
    \item
    A three-bridge knot with tunnel number $2$ has
\[ h_+ = 6 \quad \mbox{and} \quad
h_n = \begin{cases} 4 & n=1,2,3 \\ 2(n-1) & n \geq 3. \end{cases}
\]
\end{itemize}
\end{itemize}

These knots all have $h_+ = h_1+2$.

However, there are knots $K$ in $S^3$ with $h_+=h_1$.  For example, there are knots with tunnel number $1$ but no $(1,1)$--position, e.g.\ \cite[Theorem 1.3]{bowmantaylorzupan}, and so such knots have $h_1=h_+=4$.
\end{example}

\section{Uniqueness of Seifert surfaces}



\begin{theorem}\label{theo:sutmanwithnonisolatinggutsishoriprime}
    Let $(M,\gamma)$ be a sutured manifold  such that each component of $\bdry M$ contains an annular suture. If $(M,\gamma)$ has guts $(G,\gamma_G)$ that are connected, incompressibly horizontally prime, and non-isolating, then $(M,\gamma)$ is incompressibly horizontally prime.
\end{theorem}

\begin{proof}
    Since the guts are non-isolating, there exists a minimal reducing product decomposition surface $\Delta$ consisting of only product disks by Lemma~\ref{lem:nonisolatingmeansproductdisks}. 
    Now let $F$ be an incompressible horizontal surface. Then because $\bdry F$ is isotopic to the core curves of $A(\gamma)$ and $M$ is irreducible, we may isotop $F$ so that it meets each component of $\Delta$ in a horizontal arc.

The surface $\Delta$ gives the decomposition 
        \[(M,\gamma) \overset{\Delta}{\leadsto} (M',\gamma') = (P, \gamma_P) \sqcup (G, \gamma_G)\]
into a product sutured manifold $(P, \gamma_P)$ and the guts $(G, \gamma_G)$.
Here, since $\Delta$ is a collection of product disks, let us regard the decomposition of $M$ along $\Delta$ into $M'$ as the closure of $M-\Delta$ (in the path metric) where $A(\gamma')$ is the restriction of $A(\gamma)$ to ${M'}$ together with two copies of $\Delta$.  Then $(M,\gamma)$ is obtained from $(M',\gamma')$ by reidentifying the two copies of $\Delta$ in $A(\gamma')$.

    For each $Q= P, G$, set $F_Q= F \cap Q$.  
    Because $\Delta$ is a collection of product disks, we may regard $A(\gamma_Q)$ as $A(\gamma)\vert_Q$ with one or two copies of components $\Delta$ as needed.

    Any component of $F$ meeting $\Delta$ has boundary, so the closed components of $F_Q$ are closed components of $F$.  Moreover, any null-homology of closed components of $F_Q$ would pull back to a null-homology of those components of $F$.
    
    Since $F$ is incompressible, $[F]=R_+(\gamma)$, and $F$ meets each component of $\Delta$ in a single horizontal arc, the following are satisfied.  
\begin{enumerate}
    \item $F_Q$ is incompressible in $Q$,
    \item $\bdry F_Q \subset \gamma_Q$ with $\bdry F_Q$ isotopic to $\bdry R_+(\gamma_Q)$, 
   \item $[F_Q,\bdry F_Q] = [R_+(\gamma_Q), \bdry R_+(\gamma_Q)]$ in $H_2(Q,\gamma_Q)$, and
     \item no closed component of $F_Q$ is null-homologous.
\end{enumerate}
Hence $F_Q$ is an incompressible horizontal surface in $Q$. 
    (Condition (2) ensures is a decomposing surface.)

Since $G$ is incompressibly horizontally prime by assumption, 
$F_G$ is isotopic into a collar of $R(G)$. Moreover, because $G$ is connected, $F_G$ is either contained in a collar of $R_+(G)$ or $R_-(G)$, say of $R_+(G)$.  Then $F_G$ is parallel to $R_+(G)$.

   Since $F_P$ is a fiber of $P$, it is parallel to $R_+(P)$ as well.  Since $F_P \cap \Delta = F_G \cap \Delta = F \cap \Delta$, we therefore have that $F = F_P \cup F_Q$ is parallel to $R_+(M) = R_+(P) \cup R_+(G)$. Hence $F = F_P \cup F_G$ is isotopic into a collar $R_+(\gamma)$.  
   Thus $(M,\gamma)$ is incompressibly horizontally prime.
\end{proof}



\begin{cor}[Cf.\cite{kobayashi}]
\label{cor:uniqueSeifert}
    Suppose $K$ is a knot with an incompressible Seifert surface $F$ for which the complementary sutured manifold $(M_F, \gamma_F)$ has guts $(G, \gamma_G)$ that are connected, incompressibly horizontally prime, and non-isolating. Then, up to isotopy fixing $K$, $F$ is the unique incompressible Seifert surface for $K$.
\end{cor}

\begin{proof}
  Since the incompressible Kakimizu complex for $K$ is connected \cite{KakimizuComplex}, if $K$ has an incompressible Seifert surface not isotopic to $F$, then it has one disjoint from $F$.  
  So let $F'$ be an incompressible Seifert surface for $K$ that is disjoint from $F$. Since $F'$ is homologous to $F$ in the exterior of $K$, it follows that $F'$ is an incompressible horizontal surface in $(M_F, \gamma_F)$.  By Proposition \ref{theo:sutmanwithnonisolatinggutsishoriprime} $(M_F, \gamma_F)$ is incompressibly horizontally prime, then $F'$ isotopic into a collar of $R(\gamma_F)$. 
  Since $F'$ is connected it must be contained in a collar of, say, $R_+(\gamma_F)$.  Because such a collar is a product and $F'$ is incompressible, it must then be parallel to $R_+(\gamma_F)$. Since $E(K)$ is recovered by gluing back the product $F\times I$ to $(M_F, \gamma_F)$ the parallelism between $R_+(\gamma_F)$ and $F'$ extends into $E(K)$ to a parallelism between $F$ and $F'$. Thus $F$ and $F'$ are equivalent Seifert surfaces.  Hence $F$ is the unique incompressible Seifert surface for $K$.
  \end{proof}

\begin{remark}
    Corollary~\ref{cor:uniqueSeifert} also follows from Kobayashi \cite{kobayashi} in a similar, but slightly different way. 
    (Indeed, one could prove a version of Theorem~\ref{theo:sutmanwithnonisolatinggutsishoriprime} showing that a sutured manifold with non-isolating APSM guts is also an APSM.) 
    If the guts of an incompressible Seifert surface are connected, then they have connected boundary and hence the property of being incompressibly horizontally prime is equivalent to being an APSM.
    See Remark~\ref{rem:kobayashi}. When furthermore the guts are non-isolating, Lemma~\ref{lem:nonisolatingmeansproductdisks} implies the decomposition may be done by product disks and so the techniques of \cite{kobayashi} yield the same result.
\end{remark}

\begin{theorem}\label{thm:IHPconverse}
    Let $(M,\gamma)$ be a connected irreducible sutured manifold with $R(\gamma)$ incompressible. Suppose $(M,\gamma)$ has guts $(G, \gamma_G)$ that are either
\begin{itemize}
    \item not connected or
    \item not incompressibly horizontally prime.
\end{itemize}
Then $(M,\gamma)$ is not incompressibly horizontally prime.
\end{theorem}
\begin{proof}
Since $(M,\gamma)$ is connected, if $(M,\gamma) = (G,\gamma_G)$ then the hypotheses imply that $(M,\gamma)$ must not be incompressibly horizontally prime.  So assume $(M,\gamma) \neq (G,\gamma_G)$.  Thus there is a non-empty collection of product annuli $\Pi$ that decomposes $(M,\gamma)$ into $(M_\Pi, \gamma_\Pi)$ which is a union of a product sutured manifold $(P, \gamma_P)$ and its guts $(G, \gamma_G)$. 

\medskip

(1) Assume $(G, \gamma_G)$ is not connected.  
Then $(G, \gamma_G)$ has at least two components.
Let $(G_1, \gamma_1)$ be one component and let $(G_2, \gamma_2)$ be the rest.  Then $(G_1, \gamma_1)$ and $(G_2, \gamma_2)$ are joined by either a component $(P_0,\gamma_0)$ of the product $(P, \gamma_P)$ or a product annulus $\Pi_0$. In the latter case, we may take a parallel copy of the product annulus $\Pi_0$ to cut out a product solid torus which we treat as a component $(P_0,\gamma_0)$ that joins $(G_1, \gamma_1)$ and $(G_2, \gamma_2)$. In particular, there are annuli $\Pi_1$ and $\Pi_2$ of $\Pi$ along which $P_0$ meets $G_1$ and $G_2$ respectively. (In that latter case, $\Pi_1$ and $\Pi_2$ are the two copies of $\Pi_0$.)

Let $S_1$ be a parallel push-off of $R_-(\gamma_1)$  in $G_1$ and let $S_2$ be a parallel push-off of $R_+(\gamma_2)$ in $G_2$.  In the product $P$ take a surface $S_P$ which is parallel to $R_+(\gamma_P)$, so that after regluing along the annuli $\Pi$, a surface $S$ is obtained by identifying $S_1, S_2$ and $S_P$ along their boundary components in the annuli corresponding to the product annuli $\Pi$. Any component of $\bdry S_1$, $\bdry S_2$, or $\bdry S_P$ that does not meet $\Pi$ is a core curve of the annular suture that contains it.  These remaining curves are $\bdry S$ and these annular sutures are $\gamma$.  As $S_2$ and $S_P$ are parallel to $R_+(\gamma_2)$ and $R_+(\gamma_P)$ respectively, while 
$S_1$ is parallel to $R_-(\gamma_1)$ which is homologous to $R_+(\gamma_1)$, we have that $S$ is homologous to $R_+(\gamma)$.
 
 If $S$ were compressible then there would exist a compressing disk that is disjoint from the product annuli $\Pi$, because $\Pi$ is incompressible.  Hence either $S_1, S_2$, or $S_P$ would be compressible in $G_1$, $G_2$, or $P$ respectively.  But these surfaces are incompressible by construction.

Observe that any connected component  $(M', \gamma')$ of $(M_\Pi,\gamma_\Pi)$ has non-empty sutures, and hence each $R_+(\gamma')$ and $R_-(\gamma')$ are nonempty.  Then, since $M'$ is connected, any component of $R(\gamma')$ intersects a vertical arc (an arc joining $R_+$ and $R_-$) once.  Now by construction, each of $S_1$, $S_2$, and the $S_{P}$ are isotopic to $R_+$ or $R_-$ of their sutured manifolds.  Thus each component of $S_1$, $S_2$, and $S_{P}$ is intersected exactly once by some vertical arc of its sutured manifold.
Therefore, if $S'$ is a closed subsurface of $S$, it too is intersected exactly once by some vertical arc of $(M,\gamma)$.  Thus $S'$ cannot be null-homologous in $H_2(M,\gamma)$.

Hence, $S$ is an incompressible horizontal surface for $(M,\gamma)$.

 By construction $S$ is not contained in any collar of $R(\gamma)$.
  If it were, let $C$ be such a collar that contains $S$.  Then $C$ is a disjoint union of a collar $C_+$ of $R_+(\gamma)$ and $C_-$ of $R_-(\gamma)$.  Observe that $C_+$ and $C_-$ restrict to collars of 
   each $R_\pm(\gamma_1)$, $R_\pm(\gamma_2)$, and $R_\pm(\gamma_{P_0})$.   By construction, $S_1$ must be contained in $C_+ \cap G_1$ while $S_2$ must be contained in $C_- \cap G_2$ since otherwise $G_1$ or $G_2$ would not be a component of the guts. Thus $S \cap \Pi_1$ is contained in $C_+$ while $S \cap \Pi_2$ is contained in $C_-$.   Since $S_{P_0}$ is connected and meets $S_{P_0} \cap \Pi_i = S \cap \Pi_i$, it cannot be contained in either $C_+$ or $C_-$, contrary to $S$ being contained in $C$.

 
 \medskip

(2) Assume $(G, \gamma_G)$ is not incompressibly horizontally prime.
Let $S_G$ be an incompressible horizontal surface in the guts $G$ which is not contained in a collar of $R(\gamma_G)$. 
 In the product $P$ take a surface $S_P$ which is parallel to  $R_+(\gamma_{P})$. Then after regluing along the traces of the product annuli $\Pi$, a  surface $S$ is obtained by identifying $S_G$ and $S_P$ along the boundaries corresponding to the product decomposition $\Pi$. By a similar argument as above we conclude that $S$ is an incompressible horizontal surface and not contained in a collar of $R(\gamma)$.
\end{proof}

\begin{cor}\label{cor:gutsNOTconnORNOTihp}
   Suppose $K$ is a knot in $S^3$ with an incompressible Seifert surface $F$. 
   If $F$ has guts that are either not connected or not incompressibly horizontally prime, then $K$ has  another incompressible Seifert surface not isotopic to $F$.
\end{cor} 

\begin{proof}
  By Theorem~\ref{thm:IHPconverse}, the  sutured manifold $(M_F, \gamma_F)$ complementary to $F$ is not incompressibly horizontally prime.  Therefore there is an incompressible horizontal surface $F'$ in $(M_F, \gamma_F)$ homologous to $R_+(\gamma_F)$ that is not contained in a collar of $R(\gamma_F)$.
  Since $\gamma_F$ is a single annulus and $F'$ is homologous to $R_+(\gamma_F)$, it follows that upon the inclusion of $M_F$ into the exterior of $K$,  $F'$ is a union of a Seifert surface for $K$ and a collection of closed surfaces disjoint from $F$.  However, since any closed embedded surface in $S^3$ is separating, these closed surfaces would be null-homologous in $M_F$.  Therefore, since $F'$ is an incompressible horizontal surface, it cannot have any closed componets.  Hence $F'$ is a Seifert surface for $K$ disjoint from $F$.

  If $F'$ were isotopic to $F$, then they would cobound a product region in the exterior of $K$ by \cite[Lemma 3.9]{CTPforKnots}.  However that would mean that $F'$ is contained in a collar of $F$ in $S^3$, and moreover that $F'$ is contained in a collar of $R(\gamma_F)$ in $M_F$, contrary to assumption.
\end{proof}

\section{Examples of nearly fibered knots with non-unique incompressible Seifert surfaces}\label{sec:exampleofnearlyfibredknotwithtypeI}

We exhibit examples of knots in $S^3$ with isolating guts of Type I. As described in Theorem~\ref{thm:gutstructure}, such guts separate off a component of the product piece with a single connected suture. Knots with such type of guts are nearly fibered and, as we shall observe, have non-unique incompressible Seifert surfaces.

\subsection{Construction of an incompressible Seifert surface}\label{sec:constructionofincompressibleSS}
Let $K$ be a knot in $S^3$ with isolating guts of Type I.  Then $K$ is nearly fibered and therefore has a unique minimal genus Seifert surface $F$. Let $(M,\gamma)$ be the complementary sutured manifold to $F$ with product decomposing surface $\Pi$ decomposing $(M,\gamma)$ into the product manifold $(P, \gamma_P)$ and the guts $(G, \gamma_G)$.
As the guts have Type I, the decomposing product surface consists of four annuli  $\Pi= \{\Pi_0, \Pi_1, \Pi_2, \Pi_3 \}$, inducing the annular sutures $\{\gamma_0, \gamma_1, \gamma_2, \gamma_3\}$ of $G$.
As the guts are isolating, the product manifold $P$ consists of two components, say $F_2 \times [-1,1]$ whose only suture is induced from $\Pi_2$ and $F_{013}\times [-1,1]$ with sutures $\gamma$ and three induced from the product annuli $\Pi_0, \Pi_1, \Pi_3$. See Figure \ref{fig:isolatingonefour1} (Left). 

The surface $F_{013}\times \{ 0\}$  may be joined to the three surfaces 
$F_2 \times \{-\tfrac12\}$, $F_2 \times \{0\}$, and $F_2 \times \{\tfrac12\}$
by annuli $A_{23}$, $A_{12}$ and $A_{02}$ in $G$ (where $A_{ij}$ is an annulus in $G$ running between $\gamma_i$ and $\gamma_j$).  Together, these form a Seifert surface $S$ for $K$ as illustrated in Figure \ref{fig:isolatingonefour1} (Right).  We now show that $S$ has greater genus than $F$ and it is incompressible.

Observe that $\chi(S)= \chi(F_{013})+3\chi(F_2)$, while $\chi(F)=\chi(F_{013})+\chi(F_2)$. Since $F_2$ has a single boundary component, $\chi(F_2) \neq 0$.  Indeed, since $F_2$ cannot be a disk, we have $\chi(F_2) \leq -1$ and $\chi(S) \leq \chi(F)-2$.
Thus $g(S) > g(F)$.

\begin{claim}\label{claim:incompsfce}
The Seifert surface $S$ is incompressible.
\end{claim}

\begin{proof}
Suppose $S$ is compressible.
By construction, $S$ is disjoint from $F$.  As $F$ is incompressible, there would be a compressing disk for $S$ that is also disjoint from $F$.  Thus $S$ and some compressing disk are contained in $(M,\gamma)$.  Among such disks transverse to $\Pi$, let $D$ be one that minimizes $|D \cap \Pi|$.  Since  $\Pi$ is incompressible any closed curve of intersection with $D$ must bound a disk in $\Pi$.  However, chopping $D$ along an innermost such disk would yield a compressing disk intersecting $\Pi$ fewer times.   Thus $D$ and $\Pi$ must intersect only in arcs that are properly embedded in $D$ while in $\Pi$ they have their endpoints on the essential curves of $S \cap \Pi$.  Each of these arcs in $\Pi$ bounds a disk in $\Pi$ with an arc of $S$.  Chopping $D$ along an innermost one of these disks will also yield a compressing disk intersecting $\Pi$ fewer times.  Hence $D$ and $\Pi$ must be disjoint.  
However, by construction, the product decomposition surface $\Pi$ chops $S$ into incompressible pieces,  level surfaces of components of the product $P$ and longitudinal annuli of the solid torus $G$.
So a compressing disk for $S$ in $M$ must non-trivially intersect $\Pi$, a contradiction.
\end{proof}

\begin{figure}
    \centering
    \includegraphics[width = .9\textwidth]{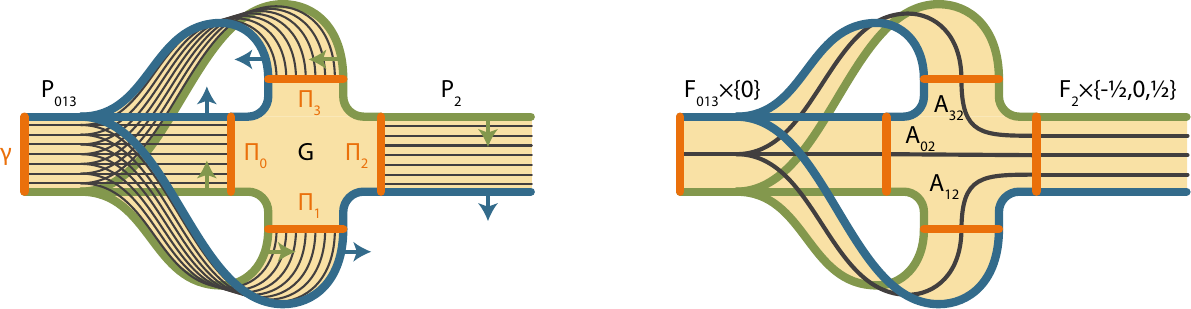}
    \caption{Left: The general structure of a sutured manifold with Type I isolating guts.   Right: The assemblage of an incompressible horizontal surface using one fiber of the non-isolated product part, three annuli in the guts, and three fibers of the isolated product part.}
    \label{fig:isolatingonefour1}
\end{figure}

\subsection{Explicit examples of knots with isolating guts of Type I}
Here we describe a construction of nearly fibered knots with isolating guts of Type I for which  $F_{013}$, the fiber of the non-isolated component of the product piece described above, is a four punctured sphere.

Let $J$ be a knot in $S^3$ with a Seifert surface $F_J$ whose complementary sutured manifold has a handle number $2$ Heegaard surface $S_J$. (So $J$ is either a fibered knot or a handle number $2$ knot.) Let $(M, \gamma)$ be the sutured manifold complementary to $S_J$.  Thus $(M, \gamma)$ is homeomorphic to the sutured manifold $(F_J \times I, \bdry F_J \times I)$ with a $1$-handle $H_+$ attached to  $F_J \times \{+1\}$ and another (dual) $1$-handle $H_-$ attached to $F_J \times \{-1\}$. Let $D_+$ and $D_-$ cocores of these two handles. Of course, as a submanifold of $S^3$, $(M, \gamma)$ is also the exterior of $(S_J \times I, \bdry S_J \times I)$.

Let $(M', \gamma')$ be obtained from $(M, \gamma) \subset S^3$ as follows.  In each $R_+(\gamma)$ and $R_-(\gamma)$ respectively, choose a properly embedded arc $a_+$ and $a_-$ that runs over its handle $H_+$ or $H_-$ (that is, crosses $D_+$ or $D_-$) exactly once.  In the annular suture $\gamma$, attach  two trivial $1$-handles $h_+$ and $h_-$ to $M$ to form the submanifold $M'$ of $S^3$.  Letting $d_+$ and $d_-$ be cocores of these handles, they are trivial in that there are dual disks $e_+$ and $e_-$ in the exterior of $M'$ so that $\bdry e_+$ and $\bdry e_-$ each meet $\gamma$ in one arc and crosses $d_+$ or $d_-$ once. (Indeed $\bdry M$ is a Heegaard surface for $S^3$ and $\bdry M'$ is the result of two Heegaard stabilizations along $\gamma$.)  Now form the core of the suture $\gamma'$ in $\bdry M'$ from two parallel copies of each $a_+$ and $a_-$ and arcs in $\gamma \cup \bdry h_+ \cup \bdry h_-$ as indicated in Figure~\ref{fig:Examplesfce}.  This produces the sutured manifold $(M', \gamma')$ as a submanifold of $S^3$ whose exterior is the complementary sutured manifold $(M'', \gamma'')$.

\begin{figure}
    \centering
    \includegraphics[width=.9\textwidth]{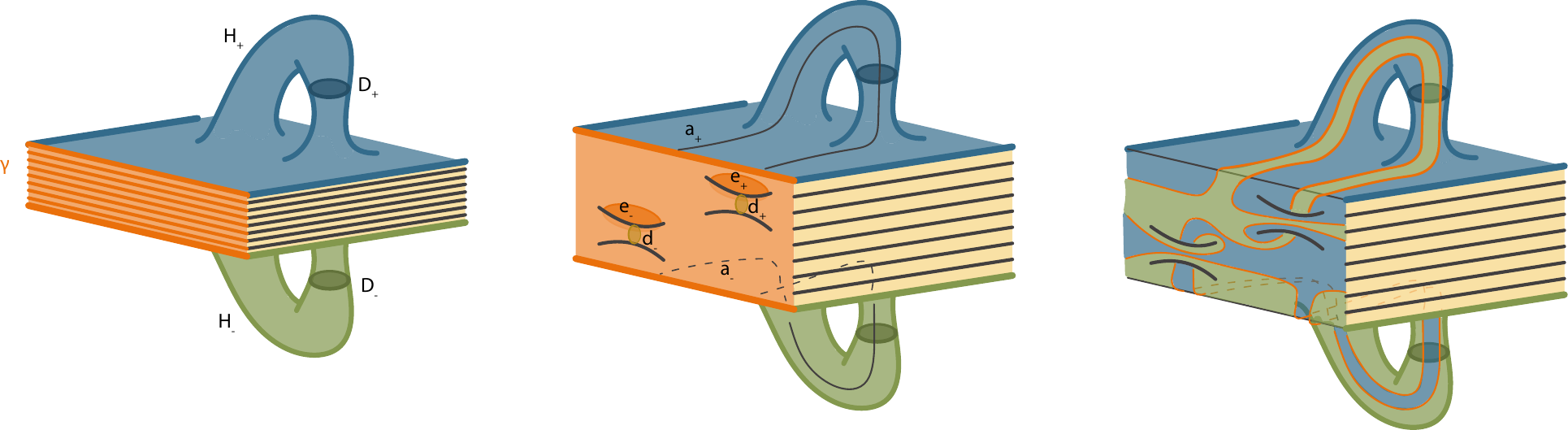}  
    \caption{Left: The exterior of a Heegaard surface giving a handle number $2$ presentation of a Seifert surface.  Handles $H_+, H_-$ with compressing disks $D_+, D_-$ are shown. Center: Arcs $a_+,a_-$ are chosen running over each of the handles once.  Two trivial handles with their pairs of compressing disks are added along the annular suture.  Right:  A new suture is placed on this manifold following the arcs $a_+, a_-$.}
    \label{fig:Examplesfce}
\end{figure}
Observe that the four disks $D_+$, $D_-$, $d_+$, and $d_-$ are product disks in $(M',\gamma')$. Decomposing $(M',\gamma')$ along those disks leaves the product sutured manifold $(F_J \times I, \bdry F_J \times I)$. 

The two disks $e_+$ and $e_-$ are product disks for the complementary sutured manifold $(M'', \gamma'')$. Decomposing $(M'', \gamma'')$ along them leaves the sutured manifold $(S_J \times I, \bdry S_J \times ([-1, -1/2] \cup [-1/4,1/4] \cup [1/2,1]))$, the product sutured manifold but with three parallel copies of its suture.  A product annulus now decomposes this into the product $(S_J \times I, \bdry S_J \times I)$ and a solid torus with four longitudinal sutures.

Then  $K= \bdry R_+(\gamma')$ is a knot in $S^3$ with Seifert surface $R_+(\gamma')$ whose complementary sutured manifold $(M'', \gamma'')$ has isolating guts of Type I.  In particular, $K$ is a nearly fibered knot, $R_+(\gamma')$ is its minimal genus Seifert surface, and it has another incompressible Seifert surface.

In Figure~\ref{fig:exampleunknot}, we apply this construction where $J$ is the unknot with a handle number $2$ Heegaard surface $S_J$ of genus $1$ for the splitting of the complement of the disk $F_J$.  Shown is the sutured manifold $(M',\gamma')$.  This knot $K$ is the core of the suture $\gamma'$ and has genus $2$. Using SnapPy \cite{snappy} with Sage \cite{sagemath}, we find that $K$ turns out to be a 16 crossing hyperbolic knot 
with 
Jones polynomial
\[-q^{-16}+q^{-14}+q^{-8}-q^{-6}+q^{-4}-q^{-2}+1-q^4+q^6\]
and Knot Floer homology ranks

\[
\begin{tabular}{r|rrrrrr}
    2  &    &    &   &   & 1 & 1\\
    1  &    &    &   & 2 & 2\\
    0  &    &    & 3 & 2 \\
    -1 &    &  2 & 2 \\
    -2 &  1 &  1 \\      
    \hline
       & -2 & -1 & 0 & 1 & 2 & 3
\end{tabular}
\]
where the horizontal \& vertical coordinates are the Maslov grading \& Alexander grading.
A referee pointed out that this is the knot 16n332130; we confirmed this with SnapPy using the arc index 11 presentation of 16n332130 found in \cite{arcindex11}.


Its incompressible Seifert surface $S$ constructed as in Section~\ref{sec:constructionofincompressibleSS} has a genus $4$.
\begin{figure}
    \centering
    \includegraphics[height=4cm]{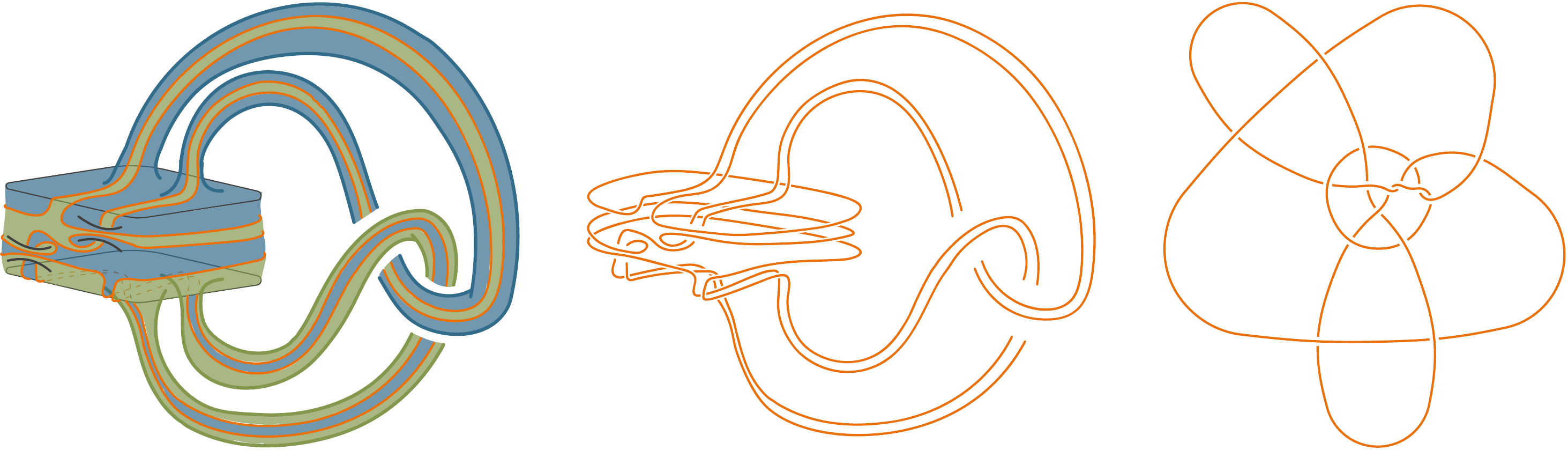}
    \caption{Left: A nearly fibered knot $K$ with isolating Type I guts is built from a handle number $2$ presentation of the unknot $J$. Center: $K$ again without the surfaces. Right: $K$ simplified to a $16$ crossing knot, as drawn in KLO \cite{KLO}.}
    \label{fig:exampleunknot}
\end{figure}




\bibliographystyle{alpha}
\bibliography{biblio}

\end{document}